\documentclass[a4paper,11pt, times, draft]{amsart}

\usepackage[utf8x]{inputenc}
\usepackage{amsmath}
\usepackage{amssymb, textcomp}
\usepackage{enumerate}
\usepackage{times}
\usepackage[varg]{txfonts}
\usepackage[mathscr]{eucal}
\usepackage{gensymb}
\usepackage[pdftex]{graphicx}
\usepackage{hyperref}

\newtheorem{theorem}{Theorem}[section]

 \newtheorem{defin}[theorem]{Definition}
 \newtheorem{remark}[theorem]{Remark}

\title{On some classes of partial difference equations}

\author{Eszter Gselmann }


\thanks{This research has been supported by the Hungarian Scientific Research Fund
(OTKA) Grant NK 81402. }

\begin{document}
\begin{abstract}
 In one of his work, appeared in 1969, John A.~Baker initiated the systematic investigation of some partial 
difference equations. The main purpose of this paper is to continue and to extend these investigations. 
Firstly, we present how such type of equations can be classified into elliptic, parabolic and hyperbolic subclasses, respectively. 
After that, we show solution methods in the elliptic class. Here we will deal in details with the 
discrete version of the following partial differential equations: Laplace's equation, Poisson equation and the (in)homogeneous biharmonic equation. 
\end{abstract}
\maketitle

\section{Introduction}
The aim of this paper is to give a systematic description of certain type of partial difference equations. 
Throughout this note $\mathbb{N}, \mathbb{Z}, \mathbb{Q}$ and $\mathbb{R}$ denote the 
sets of the natural, integer, rational and real numbers, respectively. 

Let $\Omega\subset\mathbb{R}^{n}$ be a set and $u\colon \Omega\to\mathbb{R}$ be a function, 
then 
\[
 \underset{(x_{i})}{\Delta_{h}}u(x_{1}, \ldots, x_{n})=
u(x_{1}, \ldots, x_{i-1}, x_{i}+h, x_{i+1}, x_{n})-u(x_{1}, \ldots, x_{n})
\]
is called the \emph{first order partial difference} of the function $u$ with respect to the variable $x_{i}$. 
The higher order partial differences are defined recursively, that is, 
if $k\in\mathbb{N}$ and $k\geq 2$ and $i=1, \ldots, n$, then 
\[
 \underset{(x_{i})}{\Delta^{k}_{h}}u(\textbf{x})=\underset{(x_{i})}{\Delta_{h}}\left(\underset{(x_{i})}{\Delta^{k-1}_{h}}\right)u(\textbf{x}), 
\]
assuming that $\textbf{x}+khe_{i}\in\Omega$ is valid for all $\textbf{x}\in \Omega$, where 
$e_{i}$ denotes the $i$\textsuperscript{th} standard base vector of $\mathbb{R}^{n}$. 
Further, if $k_{1}, \ldots, k_{n}\in\mathbb{N}\cup\left\{0\right\}$ and $\sum_{i=1}^{n}k_{i}=N$, then 
\[
 \underset{(x_{1}^{k_{1}}, \ldots, x_{n}^{k_{n}})}{\Delta^{N}_{h}}u(\textbf{x})=
\underset{(x_{1})}{\Delta^{k_{1}}_{h}}\ldots\underset{(x_{n})}{\Delta^{k_{n}}_{h}}u(\textbf{x}). 
\]

In view of this notions on a \emph{partial difference equation} we mean a functional equation that has the form 
\[
 \tag{1}
F\left(\underset{(x_{1}^{k_{1}}, \ldots, x_{n}^{k_{n}})}{\Delta_{h}^{N}}u(x_{1}, \ldots, x_{n}), \ldots,  u(x_{1}, \ldots, x_{n}), x_{1}, \ldots, x_{n}\right)=0, 
\]
where $F$ is a given function, $u$ is the unknown function which has to be determined and 
$k_{1}, \ldots, k_{n}\allowbreak\in\mathbb{N}\cup\left\{0\right\}$ are such that $\sum_{i=1}^{n}k_{i}=N$. 

Equation (1) is called an \emph{$N$\textsuperscript{th}-order equation} if it contains both 
$\Delta^{N}u$ and $u$. 
For the sake of simplicity, here we remark that other properties (e.g. linearity, quasilinearity, semilinearity etc.) of such type of equations can be defined 
analogously as in the theory of partial differential equations. Concerning partial difference equations we refer to the monograph of Cheng \cite{Cheng}.

\section{Classification of partial difference equations}

Let $\Omega\subset\mathbb{R}^{n}$, $m\in\mathbb{N}$ and $\gamma_{i}\colon \Omega\to\mathbb{R}$ given functions 
and $\rho_{j, i}\in\mathbb{R}$, $j=1, \ldots, n$, $i=1, \ldots, m$ and let us consider 
the operator $\mathcal{D}$ defined by 
\[
\mathscr{C}(\Omega)\ni u(x_{1}, \ldots, x_{n}) \xrightarrow{\mathcal{D}} \sum_{i=1}^{m}\gamma_{i}(x_{1}, \ldots, x_{n})\cdot 
u\left(x_{1}+\rho_{1, i}h, \ldots, x_{n}+\rho_{n, i}h\right). 
\]

Here the domain $\Omega$ is assumed to be such that $\mathbf{x}+\mathbf{\rho}_{i}h\in\Omega$ for all 
$i=1, \ldots, m$ and $\mathbf{x}=(x_{1}, \ldots, x_{n})\in\Omega$, where 
$\mathbf{\rho}_{i}=(\rho_{1, i}, \ldots, \rho_{n, i})$. 

In this case let 
\begin{multline*}
 A_{k, l}(\mathbf{x})
=A_{k, l}(x_{1}, \ldots, x_{n})
\\=
\sum_{i=1}^{m}\rho_{k, i}\rho_{l, i}\gamma_{i}(x_{1}, \ldots, x_{n}) 
\qquad 
\left(k, l=1, \ldots, n, \textbf{x}\in \Omega\right). 
\end{multline*}
Further, let us consider the (symmetric) $n\times n$ matrix valued function 
$Q$  defined by 
\[
 Q(\textbf{x})=
\begin{pmatrix}
 A_{11}(\textbf{x})& \ldots & A_{1n}(\textbf{x})\\
\vdots & \ddots & \vdots \\
A_{n1}(\textbf{x})& \ldots & A_{nn}(\textbf{x})
\end{pmatrix}
\qquad 
\left(\mathbf{x}\in\Omega\right). 
\]

We say that the operator $\mathcal{D}$ defined above is 
\begin{itemize}
 \item \emph{elliptic} at the point $\mathbf{x}\in\Omega$, if $Q(\mathbf{x})\in\mathscr{M}_{n\times n}(\mathbb{R})$ is a definite matrix;
\item \emph{hyperbolic} at the point $\mathbf{x}\in\Omega$, if $Q(\mathbf{x})\in\mathscr{M}_{n\times n}(\mathbb{R})$ is an indefinite matrix;
\item \emph{parabolic} at the point $\mathbf{x}\in\Omega$, if $Q(\mathbf{x})\in\mathscr{M}_{n\times n}(\mathbb{R})$ is a semidefinite matrix;
\end{itemize}

The systematic investigation of equations of the form 
\[
 \mathcal{D}(u)(\mathbf{x})+u(\mathbf{x})=\Phi(x)
\]
goes back to Baker \cite{Baker}, McKiernan \cite{McKiernan}, and {\'S}wiatak \cite{Swiatak}.  
The aim of this paper is to continue and to extend the results of 
John A.~Baker. Additionally,  not only the solutions will be derived but also 
we will get regularity theorems such as 'continuity implies infinitely many times differentiability'. 
Such type of results can be found e.g. in J\'{a}rai \cite{Jarai}. 

Finally, we remark that for linear, homogeneous, constant coefficient partial difference equations a 
method based on spectral synthesis was developed by L\'{a}szl\'{o} Sz\'{e}kelyhidi in \cite{Szekely}.

\section{Elliptic partial difference equations}

In this section we will investigate elliptic partial difference equations. In fact, as we shall 
see, they are nothing but the 'discrete' versions of elliptic partial differential equations. 
Here we remark that the investigations concerning hyperbolic and parabolic partial difference equations 
will appear as a continuation of this work. 

Among the most important of all partial differential equations are undoubtedly 
\emph{Laplace's equation}
\[
 \Delta u(\mathbf{x})=\displaystyle\sum_{i=1}^{n}\frac{\partial^{2}u(\mathbf{x})}{\partial x_{i}^{2}}=0 
\qquad 
\left(\mathbf{x}\in\Omega\right)
\]
and the \emph{Poisson equation}
\[
 -\Delta u(\mathbf{x})=f(\mathbf{x})
\qquad 
\left(\mathbf{x}\in\Omega\right). 
\]

In both equations the unknown function is $u\colon \overline{\Omega}\to\mathbb{R}$, 
where in general $\Omega\subset\mathbb{R}^{n}$ is assumed to be an open set. Further, 
in the second equation the function
$f\colon \Omega\to\mathbb{R}$ is also given.

\subsection{Auxiliary statements}

In this subsection we  list some preliminary tools that will be used subsequently, 
see also 
Evans \cite{Evans} and Rudin \cite{Rudin}. 

\begin{theorem}[Harnack's principle]
If the sequence functions $u_{n}\colon \Omega\to\mathbb{R}\, (n\in\mathbb{R})$  is harmonic in the domain $\Omega \subseteq\mathbb{R}^{n}$ and
$$u_1(x) \le u_2(x) \le \cdots$$
at every point of $\Omega$, then\, $\lim_{n\to\infty}u_n(x)$\, 
either is infinite at every point of the domain or it is finite at every point of the domain, in both cases 
uniformly in each closed subdomain of $\Omega$.\, 
In the latter case, the function \[u(x) = \lim_{n\to\infty}u_n(x)\qquad \left(x\in \Omega\right)\] is harmonic in the domain $\Omega$.  
\end{theorem}

\begin{remark}
Harnack's principle can be generalized to monotone sequen\-ces of solutions of elliptic equations, as well. 
\end{remark}

\begin{defin}
 We say that the function $\varphi\in\mathscr{C}^{\infty}(\mathbb{R}^{n})$ is a 
\emph{mollifier}, if 
\begin{enumerate}[(i)]
 \item it is compactly supported
\item \[\int_{\mathbb{R}^{n}}\varphi(x)dx=1\]
\item \[\lim_{\varepsilon \to 0+}\frac{1}{\varepsilon^{n}}\varphi\left(\frac{x}{\varepsilon}\right)=\delta(x) \qquad \left(x\in\mathbb{R}^{n}\right), \]
where $\delta$ denotes the Dirac delta function and the limit is understood in the Schwartz space. 
\end{enumerate}

If the function $\varphi$ is a mollifier and for all $x\in\mathbb{R}^{n}$ $\varphi(x)\geq 0$ holds, then $\varphi$ is said to be \emph{nonnegative mollifier}. \\

In case $\varphi$ is a mollifier and
there exists a function $\sigma\in\mathscr{C}^{\infty}(\mathbb{R})$ such that 
\[
 \varphi(x)=\sigma(\left\| x\right\|) 
\qquad 
\left(x\in\mathbb{R}^{n}\right), 
\]
then $\varphi$ is termed to be a \emph{(radially) symmetric mollifier}. 
\end{defin}

It is easy to see that if 
\[
 \varphi(s)=
\begin{cases}
 \exp\left(-\frac{1}{1-s}\right), & \text{ if } s<1\\
0, & \text{otherwise}
\end{cases}
\]
and 
\[
 \omega(x)=\dfrac{\varphi(\left\| x\right\| ^{2})}{\int_{\mathbb{R}^{n}}\varphi(\left\| x\right\| ^{2})dx}
\qquad
\left(x\in\mathbb{R}^{n}\right), 
\]
then the functions $\omega_{\varepsilon}\colon \mathbb{R}^{n}\to \mathbb{R}$ defined by 
\[
 \omega_{\varepsilon}(x)=\frac{1}{\varepsilon^{n}}\omega \left(\frac{x}{\varepsilon}\right) 
\qquad 
\left(x\in \mathbb{R}^{n}\right)
\]
are nonnegative, symmetric mollifiers for all $\varepsilon>0$.

\begin{theorem}[Mollification Theorem]
 Let $f\in \mathbb{L}^{1}(\mathbb{R}^{n})$, then 
\begin{enumerate}[(i)]
 \item 
$
\displaystyle \lim_{\varepsilon\to 0+}\left\|f\ast \omega_{\varepsilon} -f \right\|_{\mathbb{L}^{1}}=0
$
\item for all  $\varepsilon>0$
\[
 \mathrm{supp}(f\ast \omega_{\varepsilon})\subset \mathrm{supp}(f)+\overline{B}(0, \varepsilon)
\]
holds, where $\overline{B}(0, \varepsilon)=\left\{x\in\mathbb{R}^{n}\, \vert \, \left\| x\right\|\leq \varepsilon\right\}$. 
\item every sequence $(\varepsilon_{n})_{n\in\mathbb{N}}$ of positive real numbers 
admits a subsequence $(\varepsilon_{n_{k}})_{k\in\mathbb{N}}$ such that 
$f\ast \omega_{\varepsilon_{n_{k}}}$ converges to $f$ almost everywhere
\item if additionally, $f\in\mathscr{C}(\Omega)$, where $\Omega \subset \mathbb{R}^{n}$ is an open set, 
then 
\[
 \lim_{k\to\infty}f\ast \omega_{\varepsilon_{n_{k}}}(x)=f(x) 
\qquad 
\left(x\in \Omega\right)
\]
and the convergence is uniform on any compact subset of $\Omega$ 
\item for all $\varepsilon>0$ we have $f\ast\omega_{\varepsilon}\in\mathscr{C}^{\infty}(\mathbb{R}^{n})$, further, 
if $\alpha\in\mathbb{N}^{n}_{0}$ is a multi index, then 
\[
 \partial^{\alpha}\left(f\ast\omega_{\varepsilon}\right)=f\ast\left(\partial^{\alpha}\omega_{\varepsilon}\right). 
\]
\end{enumerate}
\end{theorem}

\subsection{The discrete Laplace's equation}

In this subsection we will investigate the \emph{discrete Lapla\-ce's equation}, that is, 
\[
\tag{$\mathcal{L}$}
\sum_{i=1}^{n} \underset{(x_{i}^{2})}{\Delta_{h}^{2}}u(\mathbf{x})=
\sum_{i=1}^{n}\left[u(\mathbf{x}+2e_{i}h)-2u(\mathbf{x}+e_{i}h)+u(\mathbf{x})\right]=0. 
\]

Let us observe that in this case, the corresponding coefficient matrix $Q(\mathbf{x})$ is the 
$n\times n$ identity matrix at every point $\mathbf{x}\in\mathbb{R}^{n}$. 
Thus equation $(\mathcal{L})$ can be considered as the 
most representative example for the class of elliptic partial difference equations. 

\begin{theorem}\label{T1}
 Let $u\colon \mathbb{R}^{n}\to\mathbb{R}$ be a continuous function  
and assume that $u$ fulfills the discrete Laplace's equation on $\mathbb{R}^{n}$. 
Then the function $u$ is harmonic on $\mathbb{R}^{n}$. 
\end{theorem}
\begin{proof}
 Assume that for all $\varepsilon>0$ the functions $\varphi_{\varepsilon}\in\mathscr{C}^{\infty}(\mathbb{R}^{n})$ 
are nonnegative, symmetric mollifiers and let 
\[
 u_{\varepsilon}(\mathbf{z})
=
\left(u\ast \varphi_{\varepsilon}\right)(\mathbf{z})
=
\int_{\mathbb{R}^{n}}u(\mathbf{x})\varphi_{\varepsilon}(\mathbf{z}-\mathbf{x})d\mathbf{x} 
\qquad 
\left(\mathbf{z}\in\mathbb{R}^{n}\right). 
\]
Due to  linearity of equation $(\mathcal{L})$ and because of the properties of the convolution, 
\[
 \sum_{i=1}^{n} \underset{(x_{i}^{2})}{\Delta_{h}^{2}}u_{\varepsilon}(\mathbf{z})=0 
\qquad 
\left(\mathbf{z}\in\mathbb{R}^{n}\right). 
\]
In other words, this means that for all $\varepsilon>0$ the functions 
$u_{\varepsilon}$ also fulfill equation $(\mathcal{L})$. Additionally, due to the 
Mollification Theorem, we also have that 
\[
 \lim_{\varepsilon\to 0+}\left\|u_{\varepsilon}-u \right\|_{\mathbb{L}^{1}}=0. 
\]
Now let $(\varepsilon_{n})_{n\in\mathbb{N}}$ be a null sequence for which the 
sequence of functions $(u_{\varepsilon_{n}})_{n\in\mathbb{N}}$ is monotone. Form this 
we get that the limit $\lim_{n\to\infty}u_{\varepsilon_{n}}(\mathbf{z})$
exists for all $\mathbf{z}\in\mathbb{R}^{n}$ and the limit function 
is a harmonic function, too. 

This means that the function $u$ is a harmonic function. 
\end{proof}

\begin{remark}
 In view of the previous theorem we immediately get that if a continuous 
function satisfies equation $(\mathcal{L})$, then it is automatically 
in the class $\mathscr{C}^{\infty}(\mathbb{R}^{n})$. 

It is obvious however, that equation $(\mathcal{L})$ has nowhere continuous solutions, as well. 
To see this, let $\alpha\colon \mathbb{R}\to\mathbb{R}$ be a nowhere continuous additive function 
(see Corollary 5.2.2 in \cite[page 130.]{Kuczma}) and 
\[
 u(\mathbf{x})=u(x_{1}, \ldots, x_{n})=\prod_{i=1}^{n}\alpha(x_{i}) 
\qquad 
\left(\mathbf{x}\in\mathbb{R}^{n}\right). 
\]
\end{remark}

Finally, we remark that the method we used in the proof of Theorem \ref{T1} 
is also appropriate to solve more general elliptic equations such as 
\[
 \dfrac{1}{h^{2}}\sum_{i=1}^{m}\gamma_{i}u(\mathbf{x}+\mathbf{\rho}_{i}h)+
u(\mathbf{x})=\Phi(\mathbf{x}). 
\]
Assume namely that is equation is elliptic. 
Then the coefficient matrix (which will be independent of 
$\mathbf{x}$ in this case), is 
\[
Q=
 \begin{pmatrix}
 A_{11}& \ldots & A_{1n}\\
\vdots & \ddots & \vdots \\
A_{n1}& \ldots & A_{nn}
\end{pmatrix}
\]
a definite matrix. Here, 
\[
 A_{k, l}=\sum_{i=1}^{m}\gamma_{i}\rho_{k, i}\rho_{l, i} 
\qquad 
\left(k, l=1, \ldots, n\right). 
\]
Thus applying the method based on the Mollification Theorem, we obtain that 
\[
 \sum_{k, l=1}^{n}A_{k, l}\dfrac{\partial^{2}u_{\varepsilon}(\mathbf{z})}{\partial x_{k}\partial x_{k}}+
u_{\varepsilon}(\mathbf{z})=\Phi_{\varepsilon}(\mathbf{z})
\]
holds for all $z\in\mathbb{R}^{n}$ and for all $\varepsilon>0$, where 
\[
 u_{\varepsilon}(\mathbf{z})
=
\left(u\ast \varphi_{\varepsilon}\right)(\mathbf{z})
\qquad 
\text{and}
\qquad 
 \Phi_{\varepsilon}(\mathbf{z})
=
\left(\Phi\ast \varphi_{\varepsilon}\right)(\mathbf{z})
\qquad 
\left(\mathbf{z}\in\mathbb{R}^{n}\right).
\]
Since the matrix $Q$ is definite, this latter partial differential equation 
is of elliptic type. Thus, after a suitable change of the variables, it can be 
transformed into its canonical form. Using the theory of elliptic partial differential 
equations, the solution of the original partial difference 
equation can be obtained easily, with the help of Harnack's principle.

\subsection{The discrete Poisson equation}

Now we consider the \emph{discrete Poisson equation}, 
that is, 
\[
\tag{$\mathcal{P}$}
\dfrac{1}{h^{2}}\sum_{i=1}^{n} \underset{(x_{i}^{2})}{\Delta_{h}^{2}}u(\mathbf{x})=f(\mathbf{x}),  
\]
where $f\colon \mathbb{R}^{n}\to\mathbb{R}$ is a given function. 

\begin{theorem}\label{T2}
Let $f\colon \mathbb{R}^{n}\to\mathbb{R}$ be a continuous, compactly supported function and 
assume that 
 $u\colon \mathbb{R}^{n}\to\mathbb{R}$ is a continuous function 
which fulfills the discrete Poisson equation on $\mathbb{R}^{n}$. 
Then 
\[
 u(\mathbf{x})=\int_{\mathbb{R}^{n}}\Phi(\mathbf{x}-\mathbf{y})f(\mathbf{y})d\mathbf{y} +C 
\qquad 
\left(\mathbf{x}\in\mathbb{R}^{n}\right), 
\]
where $C$ is a real constant and $\Phi$ stands for the fundamental solution 
of the Laplace's equation
\[
 \Phi(\mathbf{x})=\frac{1}{2\pi}\ln(\left\| \mathbf{x}\right\|)
\qquad 
\left(\mathbf{x}\in\mathbb{R}^{n}\right)
\]
 if $n=2$ and
\[
\Phi(\mathbf{x})=-\frac{1}{(n-2)s_{n}}\left\| \mathbf{x}\right\|^{-n+2}
\qquad 
\left(\mathbf{x}\in\mathbb{R}^{n}\right)
\]
if $n\geq 3$,  
where $s_{n}$ denotes the area of the $n$-dimensional unit sphere. 
\end{theorem}
\begin{proof}
Similarly as in the previous theorem, assume that for all $\varepsilon>0$ the functions $\varphi_{\varepsilon}\in\mathscr{C}^{\infty}(\mathbb{R}^{n})$ 
are nonnegative, symmetric mollifiers and let 
\[
 u_{\varepsilon}(\mathbf{z})
=
\int_{\mathbb{R}^{n}}u(\mathbf{z})\varphi_{\varepsilon}(\mathbf{z}-\mathbf{x})d\mathbf{x} 
\qquad 
\left(\mathbf{z}\in\mathbb{R}^{n}\right). 
\]
Due to the linearity of equation $(\mathcal{P})$ and because of the properties of the convolution, 
we obtain that
\[
 \sum_{i=1}^{n} \underset{(x_{i}^{2})}{\Delta_{h}^{2}}u_{\varepsilon}(\mathbf{z})=h^{2} f_{\varepsilon}(\mathbf{x})
\qquad 
\left(\mathbf{z}\in\mathbb{R}^{n}\right),  
\]
where 
\[
 f_{\varepsilon}(\mathbf{z})=(f\ast \varphi_{\varepsilon})(\mathbf{z})
\qquad 
\left(\mathbf{z}\in\mathbb{R}^{n}\right). 
\]
In other words, this means that for all $\varepsilon>0$ the functions 
$u_{\varepsilon}$ fulfill a difference equation similar to equation $(\mathcal{L})$. 
Additionally, due to the 
Mollification Theorem, we also have that 
\[
 \lim_{\varepsilon\to 0+}\left\|u_{\varepsilon}-u \right\|_{\mathbb{L}^{1}}=0. 
\]
Since for all $\varepsilon>0$, we have $u_{\varepsilon}, f_{\varepsilon}\in\mathscr{C}^{\infty}(\mathbb{R}^{n})$, 
by differentiating twice with respect to $h$ and putting $h=0$, we obtain that 
\[
 \Delta u_{\varepsilon}(\mathbf{z})=f_{\varepsilon}(\mathbf{z})
\qquad
\left(\mathbf{z}\in\mathbb{R}^{n}\right), 
\]
yielding that for all $\varepsilon>0$ the functions 
$u_{\varepsilon}$ fulfill a Poisson equation. 
Applying a version of Harnack's principle concerning elliptic 
equations finally we get that the function $u$ is a solution to the 
partial differential equation 
\[
 \Delta u(\mathbf{z})=f(\mathbf{z})
\qquad 
\left(\mathbf{z}\in\mathbb{R}^{n}\right). 
\]
The statement of the theorem follows now from Theorem 1. of Evans \cite[page 22.]{Evans}. 
\end{proof}

\subsection{The discrete biharmonic equation}

Finally, we consider a fourth order partial difference equation, namely, the 
\emph{discrete biharmonic equation}
\[
 \tag{$\mathcal{B}$} 
\dfrac{1}{h^{4}}\sum_{i=1}^{n} \underset{(x_{i}^{4})}{\Delta_{h}^{4}}u(\mathbf{x})+
\dfrac{1}{h^{4}}\sum_{\underset{i\neq j}{i, j=1}}^{n} \underset{(x_{i}^{2}x_{j}^{2})}{\Delta_{h}^{4}}u(\mathbf{x})=0. 
\]

The proof of the following result is similar to those  in the previous subsections, hence we omit it.

\begin{theorem}
 Let $u\colon \mathbb{R}^{n}\to\mathbb{R}$ be a continuous function
and assume that $u$ fulfills the discrete biharmonic equation on $\mathbb{R}^{n}$. 
Then the function $u$ is biharmonic on $\mathbb{R}^{n}$, that is, 
\[
\Delta\Delta u(\mathbf{x})=
 \sum_{i=1}^{n}\dfrac{\partial^{4}u(\mathbf{x})}{\partial x_{i}^{4}}+
\sum_{\underset{i\neq j}{i, j=1}}^{n}\dfrac{\partial^{4}u(\mathbf{x})}{\partial x_{i}^{2}x_{j}^{2}}=0
\qquad 
\left(\mathbf{x}\in\mathbb{R}^{n}\right). 
\]
\end{theorem}

\begin{remark}
 Similarly, as in case of the discrete Poisson equation, the solutions of the 
inhomogeneous biharmonic equation 
\[
 \dfrac{1}{h^{4}}\sum_{i=1}^{n} \underset{(x_{i}^{4})}{\Delta_{h}^{4}}u(\mathbf{x})+
\dfrac{1}{h^{4}}\sum_{i, j=1}^{n} \underset{(x_{i}^{2}x_{j}^{2})}{\Delta_{h}^{4}}u(\mathbf{x})=\Phi(\mathbf{x})
\]
can be derived in the same way as in the proof of Theorem \ref{T2}. 
\end{remark}


\vspace{1cm}

\noindent\textbf{Eszter Gselmann}\\
Institute of Mathematics\\
University of Debrecen\\
Egyetem t\'{e}r 1.\\
Debrecen\\
Hungary\\
{\tt gselmann@science.unideb.hu}\\

\end{document}